\documentclass[11pt]{amsart}

\title{Asymptotic Translation Length in the Curve Complex}
\author{Aaron D. Valdivia}
\email{aaron.david.valdivia@gmail.com}
\address{Florida Southern College; 111 Lake Hollingsworth Drive; Lakeland, FL 33801-5698}

\keywords{Curve complex; translation length; asymptotic; pseudo-Anosov; mapping class group}

\subjclass{30F60, 32G15}

\usepackage{amsfonts}
\usepackage{graphicx}

\newtheorem{prop}{Proposition}[section]
\newtheorem{lem}{Lemma}[section]
\newtheorem{thm}{Theorem}[section]

\newtheorem{defn}{Definition}[section]

\begin{document}

\begin{abstract}
We prove that the minimal pseudo-Anosov translation length in the curve complex behaves like $\frac{1}{\chi(S_{g,n})^2}$ for sequences where $g=rn$ for some $r\in\mathbb{Q}$.  We also show that if the genus is fixed as $n\rightarrow \infty$ then the behavior is $\frac{1}{\mid\chi_{g,n}\mid}$.  This extends results of Gadre and Tsai and answers a conjecture of theirs in the affirmative. 
\end{abstract}

\maketitle

\tableofcontents

\section{Introduction}
Let $S_{g,n}$ denote a surface with genus $g$ and $n$ punctures.  The set of homotopy classes of non-trivial simple closed curves on the surface is denoted by $\mathcal{S}$.  The curve complex $\mathcal{C}(S_{g,n})$, written $\mathcal{C}$ when there is no ambiguity, is the simplicial complex with a $0$-simplex for each element $c\in\mathcal{S}$ and an $n$-simplex for each $n-1$ tuple of disjoint elements of $\mathcal{S}$ where each $1$-simplex is given length $1$.  

The mapping class group, $\mbox{Mod}^+(S_{g,n})$, is the group of isotopy classes of orientation preserving homeomorphisms of the surface $S_{g,n}$.  The mapping class group has a natural action on the set $\mathcal{S}$ which gives an action on the curve complex $\mathcal{C}$ as a group of isometries.  The elements of $\mbox{Mod}^+(S_{g,n})$ are either periodic, reducible, or pseudo-Anosov by the Nielson-Thurston classification.  A pseudo-Anosov mapping class, $\phi$,  is a mapping class for which there exists a pair of measured singular foliations, $(\mathcal{F}^{\pm},\mu^{\pm})$ such that $\phi(\mathcal{F}^{\pm},\mu^{\pm})=(\mathcal{F}^{\pm}, \lambda^{\pm 1}\mu^{\pm})$ where $\lambda>1$ is called the dilatation.  In this paper we will investigate the pseudo-Anosov elements of the mapping class group in terms of their action on the curve complex.  We will concerned with an invariant called the asymptotic translation length of pseudo-Anosov elements in particular the minimal asymptotic translation length.  The asymptotic translation length of a pseudo-Anosov element $\phi$ is given by

$$l(\phi)=\liminf_{c\in\mathcal{S}}\lim_{n\rightarrow\infty}\frac{d(c,\phi^n(c))}{n}.$$

The minimal translation distance for a surface $S_{g,n}$ is 

$$L(S_{g,n})=\min_{\phi\in\mbox{Mod}^+(S_{g,n})}(l(\phi)).$$  

The dilatation $\lambda$ of a pseudo-Anosov mapping class is an invariant which gives the translation distance in the Teichmuller space of the surface in question.  The asymptotic translation distance is the analagous translation in the curve complex.   Our results show that the same asymptotic bounds hold for the action on the Teichmuller space and the curve complex.

We write $A$ is asymptotic with $B$ or the asymptotic behavior of $A$ is $B$ by $A\asymp B$, meaning that there is a constant $C>1$ such that $\frac{B}{C}\leq A\leq BC$.

Our results extend the work of Gadre and Tsai in \cite{gadretsai} where they prove the asymptotic behavior of the minimal translation distance for closed surfaces is $\frac{1}{\chi(S_{g,0})^2}$ where $\chi(S_{g,n})$ is the Euler characteristic of the surface $S_{g,n}$.  Our first theorem extends this result to the case $g=rn$.

\begin{thm}
If $g=rn$ such that $r\in\mathbb{Q}$ then 
$$L(S_{g,n})\asymp\frac{1}{\chi(S_{g,n})^2}.$$
\end{thm}

The proof relies on the lower bound on $L(S_{g,n})$ given in \cite{gadretsai} and an upper bound by examples constructed in \cite{valdivia}. Furthermore we prove Conjecture 6.2 of \cite{gadretsai}.  The proof of the conjecture requires a sharper lower bound for $n>>g$ which we acheive by applying some results about symmetric polynomials and techniques from \cite{Tsai} where Tsai gives a lower bound for the minimal dilatation of pseudo-Anosov mapping classes.

\begin{thm}
For $n>>g$ we have $L(S_{g,n})\geq\frac{1}{(9\alpha C+30)\mid\chi(S_{g,n})\mid-10n}$ where $\alpha C$ is a constant depending only on $g$.
\end{thm}

This lower bound along with another set of examples allows us to give the asymptotic behavior for fixed $g$ with $n$ varying.

\begin{thm}
Fixing $g>1$, as $n\rightarrow\infty$ the minimal translation length has behavior, $$L(S_{g,n})\asymp\frac{1}{\mid\chi(S_{g,n})\mid}.$$
\end{thm}

The rest of the paper is organized as follows.  In Section 2 we will discuss background material including train tracks, Lefschetz numbers, and symmetric polynomials.  In Section 3 we will give a give the proof of our lower bound, Theorem 1.2, and in Section 4 we will provide examples for the upperbounds of Theorems 1.1 and 1.3 and will then finish the proofs of these theorems. 

\noindent\textbf{Acknowledgements:}  The author is indebted to Ian Agol for establishing the proof of lemma 2.2, to Ira Gessel for helpful clarifications on symmetric polynomials, and to Nathaniel Stambaugh for helpful converstations.

\section{Background}
\subsection{Train Tracks}
A \textit{train track}, $\sigma$, is a one dimensional CW complex embedded in a surface $S_{g,n}$ with a some extra conditions attached to it.  The vertices are called \textit{switches} and the edges are called \textit{branches}.  Each branch is embedded smoothly in $S_{g,n}$ and there is a definable tangent direction at each switch for all branches meeting at that switch.  Choosing a tangent direction at each switch we can then define \textit{incoming} and \textit{outgoing} branches at each switch.  We refer the reader to \cite{P-H} for a more detailed treatment of train tracks.  

A \textit{train route} is an immersed path on $\sigma$ where at each switch the path passes from an outgoing branch to an incomming branch or vice versa.  We say that a train track $\sigma_1$ is \textit{carried} by a track $\sigma_2$, or $\sigma_1 < \sigma_2$ if there is a homotopy $f: S_{g,n}\times \mathbb{I}\rightarrow S_{g,n}$ of $S_{g,n}$ such that $f(\sigma_1,0)= \sigma_1$, $f(\sigma_1,1)\subset\sigma_2$ and each train route is taken to another train route.  A smooth simple closed curve $\gamma$ is said to be carried by $\sigma$ if the homotopy  $f: S_{g,n}\times \mathbb{I}\rightarrow S_{g,n}$  $f(\gamma,0)=\gamma$ and $f(\gamma,1)$ is a train route.

To each train track $\sigma$ with $n$ branches we can also associate the set of $n$-tuples, called \textit{measures}, of non-negative numbers $w_i$, called \textit{weights}.  Further we require that at each switch the sum of weights on the incoming branches is equal to the sum weights on the outgoing branches. We denote this set of $n$-tuples by $P(\sigma)$.  If there is a measure on $\sigma$ which is positive then $\sigma$ is called recurrent.  The set of all positive measures is $int(P(\sigma))$.

A train track is called \textit{large} if all the complementary regions are polygons or once punctured polygons.  Every pseudo-Anosov $\phi$ has a large train track, $\sigma$, such that $\phi(\sigma)<\sigma$, this train track is called an \textit{invariant} train track.  The carrying induces a transition matrix M that records the train route that each branch is taken to under the pseudo-Anosov followed by the carrying map.  The Bestvina Handel Algorithm \cite{B-H} gives one such train track, $\sigma$, and transition matrix, $M$, associated to the mapping class $\phi$.  Each of the branches of $\sigma$ fall into one of two catagories, either \textit{real} or \textit{infinitesimal}.  There are at most $9\mid\chi(S_{g,n})\mid$ real braches and at most $24\mid\chi(S_{g,n})\mid-8n$ infinitesimal branches \cite{gadretsai} (cf \cite{B-H}).

The infinitesimal branches are permuted by the mapping class while the real branches stretch over the rest of the train track.  The transition matrix, $M$, has the form

$$M=\left(\begin{array}{cc}
A&B\\
0&M_{\mathcal{R}}
\end{array}\right).$$

Here $A$ is a permutation matrix corresponding to how the infinitesimal edges are permuted.  On the other hand there is a positive integer $m$ such that $M_{\mathcal{R}}^m$ is a positive matrix.  The matrix, $M_{\mathcal{R}}$, is called the Markov partition matrix for the pseudo-Anosov, $\phi$, and keeps track of the transistion between the real edges. 

If a train track, $\sigma$, is large then a \textit{diagonal extesion} of $\sigma$ is a train track which contains $\sigma$ as a subset and the branches not in $\sigma$ meet switches at the cusps of complementary regions of $\sigma$.  We denote the set of all diagonal extensions, which is finite, by $E(\sigma)$ and 

$$P(E(\sigma))=\cup_{\sigma_i\in E(\sigma)}P(\sigma_i).$$ 

\noindent
By $int P(E(\sigma))$ we mean all measures that are positive on the braches of $\sigma$.

In their investigation of the geometry of the curve complex Masur and Minsky \cite{MM} give a nesting behavior for the measures on train tracks of a surface.  In \cite{gadretsai} Gadre and Tsai prove that the requirement that the train track be birecurrent can be replaced by only recurrent, giving the following theorem.

\begin{thm}\cite{gadretsai} (cf \cite{MM})
Given a large recurrent train track $\sigma$

$$\mathcal{N}_1(int P(E(\sigma)))\subset P(E(\sigma)).$$

Where $\mathcal{N}_1(X)$ is the 1-neighborhood of $X$.
\end{thm}

\subsection{Lefschetz numbers}

We will now review the definition and basic properties of Lefschetz numbers.  A more detailed discussion can be found in \cite{DFAT} and \cite{DT}.

If $X$ is a compact oriented manifold and $\phi:X\rightarrow X$ is a continuous map then the graph of $\phi$ is given as the set

$$graph(\phi)=\{(x,\phi(x))\mid x\in X\}\subset X\times X.$$

The diagonal of $X\times X=\Delta$ and the algebraic intersection number $i(\Delta, graph(\phi))$ is the \textit{global Lefschetz number} also denoted $Lf(\phi)$.  The Lefschetz number is invariant up to homotopy and can be computed by the trace formula,

$$\sum_{i=0}^n(-1)^i\mbox{Tr}(f_i^{\ast}),$$

where $f_i^{\ast}$ is the induced map on the homology group $H_i(X,\mathbb{R})$.

Since the Lefschetz number of a mapping class is a homotopy invariant the Lefschetz of $\phi:S_{g,n}\rightarrow S_{g,n}$ can be computed by forgetting the marked points.

\begin{lem}\cite{Tsai}
If a mapping class $\phi$ is the identity or multitwist(after forgetting the marked points) then 

$$Lf(\phi)=2-2g.$$
\end{lem}

\subsection{Symmetric Polynomials}
Here we will review the definitons of symmetric polynomials and power symmetric polynomials and develop a few key components of our proof for the new lower bound in Theorem 1.2.  For a more complete discussion of symmetric ploynomials we refer the reader to \cite{Mac}.

\begin{defn}  A partition is an $n$-tuple of non-negative integers $\lambda=(\lambda_1,\dots, \lambda_n)$ such that $\lambda_1\geq\lambda_2\geq\dots\geq\lambda_n$.  Each $\lambda_i$ is called a part of $\lambda$.  The length of $\lambda$ is denoted $l(\lambda)=n$, and the weight is the sum of the components, $\mid\lambda\mid=\displaystyle{\sum_i\lambda_i}.$
\end{defn}

The power symmetric polynomials, $p_k$, for $N$ variables, $x_1,\dots, x_N$ are defined as 

$$p_k(x)=\sum_{i=1}^Nx_i^k.$$

Furthermore we can define for any partition $\lambda$ the polynomial $p_{\lambda}=p_{\lambda_1}\dots p_{\lambda_n}.$

A symmetric polynomial in $N$ variables is a polynomial that is invariant under the action of the symmetric group $S_N$ on those $N$ variables.  These polynomials are generated by the elementary symmetric polynomials $e_n$, the sum of all products of $n$ distinct variables and we additionally define $e_0=1$.  For $n>0$ we have

$$e_n=\sum_{i_1<i_2<\dots<i_n}x_{i_1}x_{i_2}\dots x_{i_n},$$

and $e_n=0$ for $n>N$.

The generating function, $E(t)$, for $e_n$ gives a way to relate the elementry symmetric polynomails to polynomials of 1 variable.

$$E(t)=\sum_{n\geq 0}e_nt^n=\prod_{n\geq 1}(1+x_it)$$

Lastly we will need Newton's formula, 

$$ne_n=\sum_{r=1}^{n}(-1)^{r-1}p_re_{n-r},$$ 

in conjunction with another formula for $e_n$ given by (2.14') in \cite{Mac}.

$$e_n=\sum_{\mid\lambda\mid=n}\epsilon_{\lambda}z_{\lambda}^{-1}p_{\lambda}$$

Here $\epsilon_{\lambda}=(-1)^{\mid\lambda\mid-l(\lambda)}$ and $z_{\lambda}=\prod_{i\geq 1}i^{m_i}m_i!$ where $m_i=m_i(\lambda)$ is the number of parts of $\lambda$ equal to $i$.

We can obtain from a formula for $p_{N+1}$ in terms of the polynomails $p_1\dots p_n$ using Newtons formula for $n=N+1$ and the formula for $e_n$.

$$p_{N+1}=\sum_{r=1}^N\sum_{\mid\lambda\mid=N+1-r}(-1)^{2N+1-l(\lambda)}z_{\lambda}^{-1}p_{\lambda}p_r$$

\begin{lem}\cite{agolmo}
Consider the monic, reciprocal, degree $N$ polynomials $q(x)=\prod_{i=1}^N(x-\mu_i)\in\mathbb{R}[x]$.  If  $p_k(\mu)\leq \delta$ where $\delta>0$ for all $k\leq N(N+1)$ then the polynomial $q(x)$ has bounded coefficients.
\end{lem}

\begin{proof}
First we observe that there is a term $(-1)^{N+1}p_1^{N+1}$ in the expression for $p_{N+1}$.  

Now assume that $p_1\dots p_N<0$ and $p_1<-R$ for some $R>>0$.  Then each term is positive since there are $l(\lambda)+1$ power symmetric factors in each term and a factor of $(-1)^{l(\lambda)-1}$.  Then we see that $p_{N+1}>\frac{R^{N+1}}{N!}>\delta$ for large enough $R$.  If some subset of $p_2\dots p_N\leq\delta$ are positive and $p_1<-R$ then the negative terms come from terms with an odd number positive power symmetric factors.  The group of terms with $j$ positive power symmetric factors can be paired with a term which replaces each positive power symmetric factor $p_{q_j}$ with $p_1^{q_j}$, this term is positive and dominates the negative terms if $R$ is large enough.

If instead $p_k<-R$ for some $1<k\leq N$ then from above we have that $p_{k(k+1)}>\delta$ and $k(k+1)\leq N(N+1)$.

So the case we are left with is when $-R<p_1\dots p_N<\delta$.  Each monic, reciprocal polynomial of degree $N$ can be written with the elementary symmetric polynomials through the generating function as $\prod_1^N(\mu_i t-1)=\sum_0^Ne_it^i$.  In turn we can write each elementary symmetric polynomial as function of the power symmetric polynomials of whose values we have restricted to a compact set.  Therefore the polynomial $q(x)$ has bounded coefficients.  

\end{proof}

\section{The lower bound}

The proof of the lower bound for curve complex translation length follows Tsai and Gadre's proof for their lower bound but includes elements of Tsai's proof of the lower bound on minimal dilitation in order to acheive a lower bound when $n>>g$.

\begin{lem}\label{rectangleintersect}\cite{Tsai}
For any pseudo-Anosov mapping class $\phi\in\mbox{Mod}(S_{g,n})$ equipped with a Markov partition, if $Lf(\phi)<0$ then there exists a rectangle $R$ of the Markov partition such that R and $\phi(R)$ intersect(i.e. there is a one on the diagonal of the Markov partition matrix). 
\end{lem}

\begin{lem}\label{tsaiforgetfull}\cite[lemma 3.2]{Tsai}
Given $\phi\in\mbox{Mod}(S_{g,n})$ let $\hat{\phi}\in\mbox{Mod}(S_{g,0})$ be the mapping class induced by forgetting the marked points.  Then there exists a constant $0< \alpha\leq F(g)$ such that $\hat{\phi}^{\alpha}$ satisfies one of the following.

\begin{itemize}
\item (1) $\hat{\phi}^{\alpha}$ is pseudo-Anosov on a connected subsurface.
\item (2) $\hat{\phi}^{\alpha}=id$.
\item (3) $\hat{\phi}^{\alpha}$ is a multitwist.

\end{itemize}

Where the upper bound on $\alpha$, $F(g)$, is a function only of $g$.
\end{lem}

\begin{lem}
In either of cases (1), (2), or (3) we have $Lf(\hat{\phi}^{\alpha C})<0$.
\end{lem}

\begin{proof}  

We address cases 2 and 3 first.  If $\hat{\phi}^{\alpha}$ is the identity or a multitwist map then so is $\hat{\phi}^{\alpha C}$ and so $Lf(\hat{\phi}^{\alpha C})<0$ by lemma 2.1.

In case 1 we have $\phi^{\alpha}$ is a pseudo-Anosov mapping class on a connected subsurface $S_{g_0, n_0}$ such that $2g_0+n_0<2g$.  Therefore the action of $\phi^{\alpha}$ on $H_1(S_{g_0,n_0}, \mathbb{Z})$ is given by a matrix $A$ of dimension at most $2g\times 2g$.  Lemma 2.2 tells us there are finitely many monic reciprocal polynomials with roots $\mu=(\mu_1\dots\mu_2)$ such that $p_k(\mu)\leq 2$ for some $k<2g(2g+1)$.  Let that finite number of polynomials be $C$.  If the characteristic polynomial of $\phi^n$ never leaves the set of $C$ polynomials then the roots are periodic and some iterate of $\phi$ will have action on homology with all eigenvalues equal to 1.  Otherwise some iterate leaves the finite set and so there is a constant $C$ depending only on $g$ such that $\mbox{Tr}(A^C)>2$ and therefore $L(\phi^{\alpha C})<0$. 

\end{proof}

This gives a positive diagonal entry in the Markov partition's transition matrix for the mapping class $\phi^{\alpha C}$ by lemma 3.1.  The number $\alpha C$ is only dependent on the genus of the surface in question.

\begin{prop}\cite[Lemma 5.2 case 2]{gadretsai}
If $\sigma_0\in E(\tau)$ and $\mu\in P(\sigma_0)$ then in at most $j\leq 6\mid\chi(S_{g,n})\mid-2n$ iterates $\phi^j(\mu)$ is positive on some real branch of $\sigma_j\in E(\tau)$.
\end{prop} 

\begin{lem}
There exists a positive integer $k\leq (9\alpha C +30)\mid\chi(S_{g,n})\mid-10n$ such that $\phi^k(\mu)$ is positive on every branch of $\tau$ where $\mu\in P(\sigma_0)$ and $\sigma_0\in E(\tau)$.
\end{lem}

\begin{proof}

By Lemma 3.3 above we see that $\phi^{\alpha C}$ has a one on the diagonal of the transition matrix for the Markov partition.  By prop 2.4 of \cite{Tsai} we see that the Markov partition matrix for $\phi^{\alpha C r}$ is positive for some $r\leq 9\mid\chi(S_{g,n})\mid$ and by lemma 3.4 we require at most $6\mid\chi(S_{g,n})\mid-2n$ iterates to be positive on a real branch.  Therefore in $(9\alpha C+6)\mid\chi(S_{g,n})\mid-2n$ iterations we will be positive on all real branches of $\tau$.  Since there are at most $24\mid\chi(S_{g,n})\mid-8n$ infinitesimal branches we require an aditional $24\mid\chi(S_{g,n})\mid-8n$ iterations to be positive on every branch.

\end{proof}

\begin{proof}[Proof of Theorem 1.2]  let $\phi:S_{g,n}\rightarrow S_{g,n}$ be a pseudo-Anosov with invariant train track $\sigma$.  Then by Lemma 3.5 there is an iterate 

$$k\leq 9\alpha C\mid\chi(S_{g,n})\mid+30\mid\chi(S_{g,n})\mid-10n$$ 

such that given a measure $\mu$ on $\sigma_0\in E(\sigma)$, $\phi^k(\mu)\in int(P(E(\sigma)))$ giving the inclusion $\phi(PE((\sigma)))\subset int(P(E(\sigma)))$.  Then using the nesting lemma (Theorem 2.1) we get the sequence of inclusions

$$P(\sigma_{i+1})\subset int(P(E(\sigma_i)))\subset \mathcal{N}_1(int(P(E(\sigma_i))))\subset \dots$$
$$int(P(E(\sigma_2)))\subset \mathcal{N}_1(P(E(\sigma_2)))\subset int(P(E(\sigma))) \subset \mathcal{N}_1(int(P(E(\sigma))))\subset P(E(\sigma)).$$

\bigskip

Then if we choose a curve $\gamma\in\mathcal{C}(S_{g,n})\backslash P(E(\sigma))$ we have $\phi^{ik}(\gamma)\in P(E(\sigma_i))$ but not in $P(E(\sigma_{i+1}))$.  We then have $d_{\mathcal{C}}(\gamma, \phi^{ik}(\gamma))\geq i$ giving 

$$l(\phi^{k})=\liminf_{i\rightarrow \infty}\frac{d_{\mathcal{C}}(\gamma, \phi^{ik}(\gamma))}{i}\geq \liminf_{i\rightarrow \infty}\frac{i}{i}=1.$$

Then using the formula $l(\phi^{n})=n l(\phi)$ we get 

$$l(\phi)\geq \frac{1}{k}.$$

\end{proof}

This gives us a better lower bound for $n>>g$, the lower bound for Theorem 1.3.

\section{Upper bounds by example and asymptotic behavior}
In this section we will describe 2 types of examples for the upper bound.  The first are the examples defined in \cite{valdivia} for rational rays defined by $g=rn$ for $r\in\mathbb{Q}$.  The second are a series of examples giving upper bounds for rays with fixed $g$.  The second set of examples are the ones we use to answer Conjecture 6.2 of \cite{gadretsai}.

The first set of examples we will consider are called Penner sequences.  These examples are generalize the examples Penner uses to give asymptotic conditions for the minimal dilatation on closed surfaces.  Before defining a Penner sequence we will need to build some notation.  First we pick a surface $S_{g,n,b}$ with $2g-2+n>0$ where $g$ is the genus, $n$ is the number of fixed points, and $b$ is the number of boundary components.  Let $\Sigma$ be homeomorphic to the surface $S_{g,n,b}$ and then consider $\Sigma_i$ to be a homeomorphic copy of $\Sigma$ for each integer $i$ with homeomorphism 

$$h_i:S_{g,n}\rightarrow \Sigma_i.$$  

We then pick two disjoint homeomorphic subsets of the boundary components, $a^+$ and $a^-$, on $\Sigma$ which gives homeomorphic copies on $\Sigma_i$, $a_i^+$ and $a_i^-$.  Then we have orientation reversing homeomorphisms 

$$\iota_i : a_i^+\rightarrow a_{i+1}^-.$$

We can then construct a surface $F_{\infty}$ which is the collection of the $\Sigma_i$ with identifications made corresponding to the $\iota_j$.  There is homeomorphism of $F_{\infty}$ to itself given by 

$$\rho(x)=h_{i+1}(h_i^{-1}(x))$$ 

where $x\in \Sigma_i$.  We can also define the surfaces $F_m=F_{\infty}\slash \rho^m$ and the quotient map $\pi_m:F_{\infty}\rightarrow F_m$.  The map $\rho$ then pushes forward to a map $\rho_m:F_m\rightarrow F_m$ which is periodic on $F_m$.  After the construction $F_m$ may have boundary components or punctures that are left invariant by the action of $\rho_m$ this may be filled in by points or discs.

We then make a choice of 2 multicurves $C$ and $D$ on $\Sigma_1$ and multicurve $\gamma\subset \Sigma_1\cup\Sigma_2$ such that 

$$\{\rho^n(C\cup\gamma)\}_{n=-\infty}^{\infty}$$

is a multicurve and

$$\{\rho^n(C\cup\gamma\cup D)\}_{n=-\infty}^{\infty}$$ 

fills $F_{\infty}$ and intersects efficeintly.  Last given the semigroup $R(C^+,D^-)$ generated by positive Dehn twists about curves in $C$ and negative Dehn twists about curves in $D$ we pick a pseudo-Anosov word $\omega\in R(C^+,D^-)$.  By pseudo-Anosov word we mean that it is pseudo-Anosov on $\Sigma_1$.  

\begin{defn}
A Penner sequence is a sequence of mapping classes $\phi_m$ such that for some choice of $\omega\in R(C^+,D^-)$ and $\gamma$ 

$$\phi_m=\rho_m d_{\pi_m(\gamma)} \pi_m(\omega).$$
\end{defn}

Mapping classes of this form are all pseudo-Anosov \cite{valdivia} (cf. \cite{Pen88}) and there is a sequence for any sequence of surfaces with $g=rn$ for some $r\in\mathbb{Q}$.

In \cite{valdivia} the train track transition matrix for these mapping classes is also given.  Let $M$ be the train track transition matrix for the pseudo-Anosov $\phi_m$.  Then each mapping class $\phi_m^m$ has the following form as a block matrix where the $n$th block corresponds to the measures induced by $\rho_m^n(\pi_m(C\cup D\cup\gamma))$. 

$$M^m=\left(\begin{array}{ccccccc}
A& D& 0& 0& .& 0& F\\
B& E& G& 0& .& 0& F^2\\
0& F& H& G& .& 0& 0\\
.& 0& F& H& .& 0& .\\
.& .& 0& F& .& 0& .\\
.& .& .& 0& .& G& .\\
.& .& .& .& .& H& G\\
C& 0& 0& 0& .& F& H
\end{array}\right)$$

Where $M^m$ is a $m\times m$ block matrix of $r\times r$ blocks.  Given a standard basis vector $e_i$ such that $r(n-1)<i\leq rn$ and $n\neq 1$ or $m$, $M^me_i$ is the sum of basis vectors $\{e_j\}$ such that $r(n-2)< j\leq r(n+1)$.  This means that if we pick a standard basis vector $e_k$ such that if $m>3$ is even $r(\lfloor\frac{m}{2}\rfloor-1)<k \leq\lfloor\frac{m}{2}\rfloor$ then we have $M^{\lfloor\frac{m}{2}\rfloor-1}e_k$ is zero in the last $r$ entries.  If $m>3$ is odd then we pick $e_k$ such that $r\lfloor\frac{m}{2}\rfloor<k \leq\lceil\frac{m}{2}\rceil$ and we also then get that $M^{\lfloor\frac{m}{2}\rfloor-1}e_k$ is zero in the last $r$ entries.  Using this fact we see that 
$$d_C(e_k, \phi_m^{m(\lfloor\frac{m}{2}\rfloor-1)}(e_k))\leq 2$$
 
Therefore we have $l(\phi_m^{m(\lfloor\frac{m}{2}\rfloor-1)})\leq 2$ and 

$$l(\phi_m)\leq \frac{2}{m(\lfloor\frac{m}{2}\rfloor-1)}\leq \frac{4}{m^2-2m}.$$

This gives an upper bound for all rational rays through the origin finishing the proof of Theorem 1.1.

The second set of examples is simpler.  If you consider the curves in figure 2 with numerical labeling $c_1\dots c_{2g+n}$ from left to right we can easliy see that the mapping class, $\psi_{g,n}$ defined by a positive or negative Dehn twist about each curve $c_i$ starting with $c_1$ and ending with $c_{2g+n}$ where we perform a positive twist about each odd curve and a negative twist about each even curve then we see that $d_C(\psi_{g,n}^n(c_{2g+n}),c_{2g+n})=2$ giving $l(\psi_{g,n})\leq\frac{2}{n}$.  This gives the upper bound for Theorem 1.3 which completes the proof.  

\begin{figure}
\centering\includegraphics{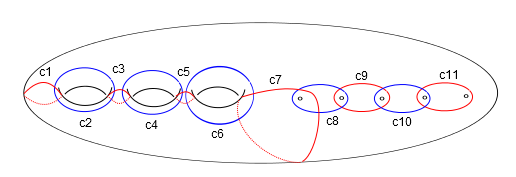}
\caption{Curves $c_i$ for the mapping class $\psi_{3,5}$.}
\end{figure}

\bibliography{mybib}
\bibliographystyle{alpha}

\end{document}